\documentclass[a4paper,12pt]{article}
\usepackage{amssymb,amsthm,amsmath,amsfonts}
\usepackage{graphicx}
\usepackage{cite}
\usepackage{color}
 \newtheorem{thm}{Theorem}[section]
 \newtheorem{cor}[thm]{Corollary}
 
 \newtheorem{prop}[thm]{Proposition}
 \theoremstyle{definition}
 \newtheorem{defn}[thm]{Definition}
 \theoremstyle{remark}
 \newtheorem{rem}[thm]{Remark}
 
 \numberwithin{equation}{section}
\usepackage{color}

\begin{document}
\title{Some characterizations of slice regular Lipschitz type spaces}
\small{
\author
{Jos\'e Oscar Gonz\'alez-Cervantes$^{(1)}$, Daniel Gonz\'alez-Campos$^{(2)}$,\\ Juan Bory-Reyes$^{(3)\footnote{corresponding author}}$}
\date{\small {$^{(1)}$ ESFM-Instituto Polit\'ecnico Nacional. 07338, Ciudad M\'exico, M\'exico\\ Email:jogc200678@gmail.com\\
$^{(2)}$ SEPI-ESFM-Instituto Polit\'ecnico Nacional. 07338, Ciudad M\'exico, M\'exico\\Email: daniel\_uz13@hotmail.com\\
$^{(3)}$ SEPI-ESIME-Zacatenco-Instituto Polit\'ecnico Nacional. 07338, Ciudad M\'exico, M\'exico\\Email:juanboryreyes@yahoo.com}}

\maketitle

\begin{abstract}
We give some characterizations of Lipschitz type spaces of slice regular functions in the unit ball of the skew field of quaternions with prescribed modulus of continuity.
\end{abstract}

 
\section{Introduction}
The quaternionic valued functions of a quaternionic variable, often referred to as slice regular functions, was born in \cite{GS1,GS2}. This class of functions, which would somehow resemble the classical theory of holomorphic functions of one complex variable, has been studied extensively in the last years, see \cite{ACS, CSS, CGSS,CSS2, GS, GenSS} and the references given there.
		 
It was shown in \cite{D,P} some characterization of generalized Lipschitz type spaces of holomorphic functions with prescribed behavior near the
unit circle centered at the origin, determined by a regular majorant in terms of the moduli of their members. Rather surprisingly, several authors attempted to extend the aforementioned characterizations to (holomorphic) smoothness spaces of both complex and vector-valued functions, see \cite{D1, D2, P1, DP, BKR, HS} and the references therein.

This paper is devoted to establish some analogous results of Dyakonov's paper \cite{D} for the theory of slice regular quaternionic functions. The main results are Proposition 3.5 and Corollary 3.15 of Proposition 3.14 as well as Corollary 3.18 of Proposition 3.17.  

\section{Preliminaries}\label{pre}
\subsection{Antecedents in standard complex analysis}
For the convenience of the reader, we recall the relevant material from \cite{D,P} and provide some additional notations and terminology, thus making our exposition self-contained.

Let $\mathbb D$ stand for the unit disc in the complex plane $\mathbb C$, $\mathbb S^1$ be the unit circle and $\overline{\mathbb D}:=\mathbb D\cup \mathbb S^1$. The algebra $Hol(\mathbb D)$ consists of those holomorphic functions on $\mathbb D$ that are continuous up to $\mathbb S^1$.

A continuous function $\omega: [0, 2] \rightarrow \mathbb R_{+}$ with $\omega(0)=0$ will be called a regular majorant if $\omega(t)$ is increasing, $\displaystyle\frac{\omega(t)}{t}$ is decreasing for $t\in [0, 2]$ and such that 
$$ \int_0^x \frac{\omega(t)}{ t} dt + x \int_x^2 \frac{\omega(t)}{ t^2}dt \leq C\omega(x), \quad  0<x<2.$$
Here and subsequently $C$ stands for a positive real constant, not necessarily the same at each occurrence. When necessary, we will use subscripts to differentiate several constants.

Given a regular majorant the Lipschitz type space, denoted by $\Lambda_\omega (\mathbb D)$, consists (by definition) of all complex valued functions $f$ defined on $\mathbb D$ such that 
$$|f(z)-f(\zeta)| \leq C \omega (|z-\zeta|) , \quad  \forall  z,\zeta \in \mathbb D.$$
The class $\Lambda_{\omega} (\mathbb S^1)$, is defined similarly.

Let us state the main results of \cite{D} as Theorems A and B, the proofs of which were considerably shorted in \cite{P}.
 
Theorem A. Let $\omega$ be a regular majorant. A function $f$ holomorphic in $\mathbb D$ is in $\Lambda_{\omega} (\mathbb D)$ if and only if so is its modulus $|f|$.

If $f\in Hol(\mathbb D)$, then $|f|$ is a subharmonic function, hence the Poisson integral of $|f|$, denoted by $P[|f|]$, is equal to the smallest harmonic majorant in $\mathbb D$. In particular, $P[|f|]-|f|\geq 0$ in $\mathbb D$.

Theorem B. Let $\omega$ be a regular majorant, $f\in Hol(\mathbb D)$, and assume the boundary function of $|f|$ belongs to $\Lambda_{\omega}(\mathbb S^1)$. Then $f$ is in $\Lambda_w (\mathbb D)$ if and only if
$$P[|f|](z)-|f(z)| \leq C \omega (1-|z|).$$

The following notation will be needed
\begin{align*}
\| f\|_{\Lambda_{\omega} (\mathbb D)} = \sup\{ \frac{|f(z)- f(\zeta)|}{\omega(|z-\zeta|)} \  \mid \ z, \zeta\in \mathbb D, \ z\neq \zeta\}, \quad \forall f\in C(\overline{\mathbb D}, \mathbb C).
\end{align*}  
The notation ${\mathfrak A} \asymp {\mathfrak B}$ means that there exist positive constants $C_1$ and $C_2$ such that $C_1 {\mathfrak A} \leq {\mathfrak B} \leq C_2 {\mathfrak A}$.

Let $\omega$ be a majorant and $f\in Hol(\mathbb D) \cap  C(\overline{\mathbb D}, \mathbb C)$. We introduce the following notations:
\begin{align*}
N_1(f) := &  \|  \  |f| \  \|_{\Lambda_{\omega} (\mathbb S^1)} + \sup\{   \frac{ P[|f|] (z) - |f|(z)  }{ \omega( 1- |z|) } \  \mid \ z\in \mathbb D\}, \\
N_2(f) := &  \|  \  |f| \  \|_{\Lambda_{\omega} (\mathbb S^1)} + \sup\{   \frac{ | \ |f|  (\zeta) - |f|(r \zeta ) \ | }{ \omega( 1- r) } \  \mid \ \zeta\in \mathbb S^1, \  0\leq r < 1 \} ,\\
N_3(f) := &  \|  \  |f| \  \|_{\Lambda_{\omega} ( \overline{\mathbb D} )}.
\end{align*}
In particular, we have: 
\begin{enumerate}
\item  If $\omega$ and $\omega^2$ are regular majorants then 
\begin{align}\label{Result1OfD}
\|f\|_{\Lambda_{\omega}(\mathbb D)} \asymp \sup\{\frac{\left\{ P[|f|^2] (z) - |f(z) |^2\right\}^{\frac{1}{2} } 
  }{ \omega( 1- |z|) } \  \mid \ z\in \mathbb D\}.
\end{align} 
\item If $\omega$ is a regular majorant then 
\begin{align}\label{Result2OfD}
\|f\|_{\Lambda_{\omega}(\mathbb D)} \asymp N_1(f) \asymp N_2(f) \asymp N_3(f),
\end{align} 
\end{enumerate} 
for any $f\in Hol(\mathbb D) \cap C(\overline{\mathbb D}, \mathbb C)$.

\subsection{Brief introduction to  slice regular functions}
A quaternion is given by $q=x_0  + x_{1} {e_1} +x_{2} e_2 + x_{3} e_3$ where $x_0, x_1, x_2, x_3$ are real values and  the  imaginary units satisfy:  
$e_1^2=e_2^2=e_3^2=-1$,  $e_1e_2=-e_2e_1=e_3$, $e_2e_3=-e_3e_2=e_1$, $e_3e_1=-e_1e_3=e_2$. The skew field of quaternions is denoted by $\mathbb H$. The sets $\{e_1,e_2,e_3\}$ and $\{1,e_1,e_2,e_3\}$ are called the standard basis of $\mathbb R^3$ and $\mathbb H$, respectively. 
The vector part of $q\in \mathbb H$ is  ${\bf{q}}= x_{1} {e_1} +x_{2} e_2 + x_{3} e_3$ and its real part is $q_0=x_0$.  
The  quaternionic conjugation  of $q$  is   $\bar q=q_0-{\bf q} $ and its  norm  is  $\|q\|:=\sqrt{x_0^2 +x_1^2+x_2^3+x_3^2}= \sqrt{q\bar q} = \sqrt{\bar q  q}$.

By abuse of notation, the unit open ball in $\mathbb H$ will be denoted by $\mathbb D^4:=\{q \in \mathbb H\ \mid \ \|q\|<1 \}$ so will the unit spheres in $\mathbb R^3$ (in $\mathbb H$) by $\mathbb{S}^2:=\{{\bf q}\in\mathbb R^3  \mid \|{\bf q}\| =1\}$ ($\mathbb{S}^3:=\{{q}\in\mathbb H \mid  \|{q}\| =1\}$), respectively.
   
The quaternionic structure allows us to see that  ${\bf i}^{2}=-1$, for every ${\bf i}\in \mathbb{S}^2$. Then $\mathbb{C}({\bf i}):=\{x+{\bf i}y; \ |\ x,y\in\mathbb{R}\}\cong \mathbb C$ as fields, and any $q\in \mathbb H \setminus \mathbb R$ may be rewritten by $x+ {\bf I}_q y $ where $x, y\in \mathbb R$ and ${\bf I}_q:=\| {\bf q}\|^{-1}{\bf q}\in \mathbb S^2$; i.e., $q\in \mathbb{C}({{\bf I}_q})$. Note that $q\in \mathbb R$ belongs to every complex plane.

Given $u\in \mathbb S^3$, the mapping ${\bf q} \mapsto u{\bf q}\bar u$ for all ${\bf q}\in \mathbb R^3$ is a quaternionic rotation  that preserves $\mathbb R^3$, see \cite{HJ}. For any ${\bf i} \in \mathbb S^2$ we will write 
$${\mathbb D}_{\bf i}:= \mathbb D^4 \cap \mathbb{C}({\bf i})$$ 
and 
$${\mathbb S}_{\bf i}:= \mathbb S^2 \cap \mathbb{C}({\bf i}).$$ 
Now, we recall few aspects of the slice regular functions theory of \cite{CGS3, CGSS, CSS, CSS2, GenSS}.
\begin{defn}
Let $\Omega\subset\mathbb H$ be an open domain. A real differentiable function $f:\Omega\to \mathbb{H}$ is called (left) slice regular function on $\Omega$ if   
\begin{align*}
\overline{\partial}_{{\bf i}}f\mid_{_{\Omega\cap \mathbb C({\bf i})}}:=\frac{1}{2}\left (\frac{\partial}{\partial x}+{\bf i} \frac{\partial}{\partial y}\right )f\mid_{_{\Omega\cap \mathbb C({\bf i})}}=0 \ \  \textrm{on $\Omega_{\bf i}:=\Omega\cap \mathbb C({\bf i})$,}
\end{align*}
for all ${\bf i}\in \mathbb{S}^2$ and its derivative, or Cullen derivative, see \cite{GS1}, is given by  
$$f'=\displaystyle {\partial}_{{\bf i}}f\mid_{_{\Omega\cap \mathbb C({\bf i})}} = \frac{\partial}{\partial x} f\mid_{_{\Omega\cap \mathbb C({\bf i})}}= \partial_xf\mid_{_{\Omega\cap \mathbb C({\bf i})}}.$$
\end{defn}
Let $\mathcal{SR}(\Omega)$ denote the right linear space of slice regular functions on $\Omega$.

\begin{defn}
A set $U\subset\mathbb H$ is called axially symmetric if $x+{\bf i}y \in U$ with $x,y\in\mathbb R$, then $\{x+{\bf j}y \ \mid  \ {\bf j}\in\mathbb{S}^2\}\subset U$ and $U\cap \mathbb R\neq \emptyset$. 
A domain $U\subset\mathbb H$ is called slice domain, or s-domain, if $U_{\bf i} = U\cap \mathbb C({\bf i})$ is a domain in $\mathbb C({\bf i})$  for all ${\bf i}\in\mathbb S^2$.
\end{defn}
Let $\Omega\subset\mathbb H$ an axially symmetric s-domain. A function $f\in \mathcal{SR}(\Omega)$ is said to be intrinsic if $f(q)=\overline{f(\bar q)}$ for all $q\in \Omega$. The real linear space of intrinsic slice regular functions on $\Omega$ will be denoted by $\mathcal N(\Omega)$, see \cite{CGS3,GS}. We will denote by $Z_f$ the set of zeroes of function $f$.
\begin{thm}\label{1.3}
Let $\Omega \subset\mathbb{H}$ be an axially symmetric s-domain and  $f\in\mathcal{SR}(\Omega)$.
\begin{enumerate}
\item (Splitting Property) For every ${\bf i}, {\bf j}\in \mathbb{S}$, orthogonal to  each other, there exist holomorphic functions $F, G:\Omega_{\bf i} \rightarrow \mathbb{C}({{\bf i}})$ such that $f_{\mid_{\Omega_{\bf i}}} =F +G {\bf j}$ on $\Omega_{\bf i}$, see \cite{CSS}.
\item (Representation Formula) For every  $q=x+{\bf I}_q y \in \Omega$ with $x,y\in\mathbb R$ and  ${\bf I}_q \in \mathbb S^2$ the following identity holds 
\begin{align*}
f(x+{\bf I}_q y) = \frac {1}{2}[   f(x+{\bf i}y)+ f(x-{\bf i}y)]
+ \frac {1}{2} {\bf I}_q {\bf i}[ f(x-{\bf i}y)- f(x+{\bf i}y)],
\end{align*} 
for all ${\bf i}\in \mathbb S^2$, see  \cite{CGSS}.
\end{enumerate}
\end{thm}
In the case of $\Omega={\mathbb D^4}$ and given $f, g\in \mathcal {SR}(\mathbb D^4)$ there exist two sequences of quaternions $(a_n)$ and $(b_n)$ such that 
$$f(q) = \sum_{n=0}^{\infty} q^n a_n,\  g(q) = \sum_{n=0}^{\infty} q^n b_n.$$ 
The product $f*g$ is defined as $f*g(q):=\sum_{n=0}^{\infty} q^n \sum_{k=0}^n  a_k b_{n-k}$ for all $q\in {\mathbb D^4}$. 

For $f(q)\neq 0$ the following property holds 
\begin{align*} 
 f * g(q)= f(q)g(f(q)^{-1} qf(q)),
\end{align*}
see \cite{CSS}. What is more, if $f^s$ has no zeroes, the $*$-inverse of $f$ is given by 
$$f^{-*} =\displaystyle \frac{1}{f^s} * f^c$$ 
and 
$$(f^{-*})'= - f^{-*} * f'* f^{-*},$$ 
where $f^c(q):= \sum_{n=0}^{\infty} q^n \overline{a_n}$ for all $q\in {\mathbb D^4}$ and $f^s :=  f* f^c =  f^c * f$, see \cite{csTrends, CSS}.

\section{Main results}
\begin{defn}\label{def1} Let $\omega$ be a regular majorant and ${\bf i} \in \mathbb S^2$. The set of all functions $f:\mathbb D^4 \to \mathbb H$ such that  
$$ \|f(x)-f(y)\| \leq C \omega (\| x - y \|) , \quad  \forall  x,y \in {\mathbb D}_{\bf i}$$
will be denoted by ${}_{\bf i} \Lambda_{\omega}(\mathbb D^4).$
 
We write ${}_{\bf i} \Lambda_{\omega}(\mathbb S^3)$ for the set of all functions $f:\mathbb S^3 \to \mathbb H$ such that  
$$ \|f(x)-f(y)\| \leq C \omega (\| x - y \|) , \quad  \forall  x,y \in {\mathbb S}_{\bf i}.$$
The norm of a function $f\in {}_{\bf i}\Lambda_\omega(\mathbb D^4)$ is defined as 
\begin{align*}
\|f\|_{ {}_{\bf i}\Lambda_\omega(\mathbb D^4)} = \sup\{\frac{  \|f(x)- f(y) \|  }{ \omega(\|x-y\|)}  \      \mid  \   x,y \in \mathbb D_{\bf i}, \  x\neq y\}.
\end{align*}
\end{defn}

\begin{defn}\label{def2}
Let $\omega_1 ,\omega_2$ be regular majorants and ${\bf i}, {\bf j} \in \mathbb S^2$ orthogonal to each other. We write ${}_{\bf i} \Lambda_{\omega_1, \omega_2}(\mathbb D^4)$ for the set of all functions $f:\mathbb D^4 \to \mathbb H$ such that    
$$\|f_k(x)-f_k(y)\| \leq C_k \omega_k (\| x - y \|) , \quad  \forall  x,y \in {\mathbb D}_{\bf i},$$
for $k=1,2$, where $f\mid_{\mathbb D_{\bf i}} = f_1+f_2 {\bf j}$ with $f_1,f_2:{\mathbb D}_{{\bf i}}\to \mathbb C({\bf i})$.  The set  
${}_{\bf i} \Lambda_{\omega_1, \omega_2}(\mathbb S^3)$ consists of all $f:\mathbb S^3\to \mathbb H$ such that    
$$\|f_k(x)-f_k(y)\| \leq C_k \omega_k (\| x - y \|) , \quad  \forall  x,y \in {\mathbb S}_{\bf i},$$
for $k=1,2$, where $f\mid_{\mathbb D_{\bf i}} = f_1+f_2 {\bf j}$ with $f_1,f_2:{\mathbb S}_{{\bf i}}\to \mathbb C({\bf i})
$. For $f\in {}_{\bf i}\Lambda_{\omega_1, \omega_2} (\mathbb D^4)$ we define 
\begin{align*}
\|f\|_{ {}_{\bf i}\Lambda_{\omega_1, \omega_2}(\mathbb D^4)}^2 = \sup\{    \frac{  \|f_1(x)- f_1(y) \|^2  }{ \omega_1(\|x-y\|)^2} +  \frac{  \|f_2(x)- f_2(y) \|^2  }{ \omega_2(\|x-y\|)^2}  \      \mid  \   x,y \in \mathbb D_{\bf i}, \  x\neq y\}.
\end{align*}
Given  ${\bf i}\in \mathbb S^2$, the ${\bf i}$-Poisson integral of  $u\in C({\mathbb S}_{\bf i}, \mathbb R) $ is
$$P_{{\bf i}}[u](q) =\frac{1}{2\pi} \int_0^{2\pi} u(e^{{\bf i} t}) \frac{ 1 -\|q\|^2}{ \|q- e^{{\bf i} t} \|^2} dt, \quad q\in \mathbb D^4.$$
\end{defn}

\begin{rem}
\begin{enumerate}
\item  Let $f\in  \mathbb D^4 \to \mathbb H$ and $f=f_1+f_2{\bf j}$ on ${\mathbb D}_{\bf i}$ with $f_1, f_2: {\mathbb D}_{\bf i} \to {\mathbb D}_{\bf i} $. Then   
\begin{align*}
2f_1 = f-{\bf i}f {\bf i }, \ \  2f_2 {\bf j} = f+{\bf i}f {\bf i }, \  \ \textrm{on } \ {\mathbb D}_{\bf i},
\end{align*}
where ${\bf i}, {\bf j}\in \mathbb S^2$ are orthogonal to each other. 

If ${\bf j}'$ is another orthogonal vector to ${\bf i}$ and $f=g_1+g_2{\bf j}'$ on ${\mathbb D}_{\bf i}$ then 
  $f_1=g_1$, $f_2  =- g_2  {\bf j}'\overline{ {\bf j}} $ and due to the usage of the quaternionic norm in the previous definitions we see that theses do not depends of the choose of ${\bf j}$, since 
$$\|f_2(x)-f_2(y)\|= \|g_2(x)-g_2(y)\|, \quad \forall x,y\in\mathbb D_{\bf i}.$$
\item Let $\omega, \omega_1 ,\omega_2$ regular majorants and $f:\mathbb D^4 \to \mathbb H$  with  $f\mid_{\mathbb D_{\bf i}} = f_1+f_2{\bf j},$ where $f_1,f_2:{\mathbb D}_{{\bf i}}\to \mathbb C({\bf i})$, with ${\bf i}, {\bf j} \in \mathbb S^2$ are orthogonal to each other. Due to inequalities  
\begin{align*}
 \left. \begin{array}{l} \|f_1(x)-f_1(y)\| \\ \|f_2(x)-f_2(y)\|\end{array}  \right\} \leq \|f(x)-f(y)\| \leq \|f_1(x)-f_1(y)\|  + \|f_2(x)-f_2(y)\| , 
\end{align*}
for all $x,y\in \mathbb D_{\bf i}$, we get that  
$${}_{\bf i} \Lambda_{\omega , \omega }(\mathbb D^4) =  {}_{\bf i} \Lambda_{\omega}(\mathbb D^4)$$ 
and    
$${}_{\bf i} \Lambda_{\omega_1, \omega_2}(\mathbb D^4) \subset   {}_{\bf i} \Lambda_{\omega_1+ \omega_2}(\mathbb D^4).$$  
Similar relationships are obtained for ${}_{\bf i} \Lambda_{\omega}(\mathbb S^3)$ and ${}_{\bf i} \Lambda_{\omega_1, \omega_2}(\mathbb S^3)$.
\end{enumerate}
\end{rem} 

\begin{defn} \label{def3} The symbol $G \Lambda_{\omega}(\mathbb D^4)$ stands for the set of all quaternionic-valued functions $f$ defined on $\mathbb D^4$ such that  
$$\|f(x)-f(y)\| \leq C \omega (\|x- y \|) , \quad  \forall  x,y \in  {\mathbb D}^4.$$
Meanwhile, $G \Lambda_{\omega}(\mathbb S^3)$ denotes the set of all quaternionic-valued functions $f$ defined on $\mathbb S^3$ such that  
$$ \|f(x)-f(y)\| \leq C \omega (\|x- y\|) , \quad  \forall  x,y \in {\mathbb S}^3.$$
\end{defn}

\begin{prop}\label{cor111} Let $\omega_1 ,\omega_2$ regular majorants and ${\bf i} \in \mathbb S^2$. Then     
$$  {}_{\bf i} \Lambda_{\omega_1, \omega_2}(\mathbb D^4)\cap \mathcal {SR}(\mathbb D^4) \  \subset \   G \Lambda_{\omega_1+ \omega_2}(\mathbb D^4)\cap \mathcal {SR}(\mathbb D^4) \ \subset \ {}_{\bf i} \Lambda_{\omega_1 + \omega_2}(\mathbb D^4)\cap \mathcal {SR}(\mathbb D^4) .$$ 
\end{prop}
\begin{proof} The relationship  $   G \Lambda_{\omega_1+ \omega_2}(\mathbb D^4)\cap \mathcal {SR}(\mathbb D^4)\subset   {}_{\bf i} \Lambda_{\omega_1 +\omega_2}(\mathbb D^4)\cap \mathcal {SR}(\mathbb D^4) $ is a direct consequence of Definition \ref{def3}.

On the other hand, we shall see that given $f\in {}_{\bf i} \Lambda_{\omega_1, \omega_2}(\mathbb D^4)\cap \mathcal {SR}(\mathbb D^4)$ there exists a constant $L>0$ such that  
\begin{align*}
 \|f (p)-f (q)\|\leq L \sum_{k=1}^2    \omega_k (\|p-q\|) , \quad \forall p,q\in \mathbb D^4.
\end{align*}
Given $ p,q\in \mathbb D^4$ consider the following  cases:
\begin{enumerate}
\item Suppose that ${\bf p}$ and ${\bf q}$ are both the zero vector. By the Splitting Property we get 
\begin{align*}
 \|f (p)-f (q)\|\leq  & \|f_1 (p)-f_1 (q)\| + \|f_2 (p)-f_2 (q)\|  \\
               \leq  & C_3\left(  \omega_1 (\| p - q \|) +  \omega_2 (\| p - q \|)  \right),	
\end{align*}
where $C_3= \max\{C_1, C_2\}$.          
\item Suppose ${\bf p}$ is not the zero vector while ${\bf q}$ is. Consider $z= p_0 + {\bf i} |{\bf p}|$ and $\zeta= q = \bar \zeta $. Combining the Representation Formula with Splitting Property we obtain     
\begin{align*}
 2\|f (p)-f (q)\|= &\| (1-\frac{{\bf p}}{\|{\bf p}\|} {\bf i}) (f(z)- f(\zeta)) + (1+\frac{{\bf p}}{\|{\bf p}\|} {\bf i}) (f(\bar z ) - f(\bar \zeta))   \| \\
\leq & 2\| f(z)- f(\zeta)\| + 2 \|f(\bar z ) - f(\bar \zeta)   \|  \\
\leq &  2   \sum_{k=1}^2   ( \|   f_k(z)- f_k(\zeta)\| +  \|f_k(\bar z ) - f_k(\bar \zeta)   \|  ) 
\\
 \leq & 4C_3 \left(   \omega_1 (\| p-q  \|) +  \omega_2 (\| p - q \|) \right) . 
		\end{align*}
where $\| z-\zeta  \| = \|p-	q\|$ is used.
\item Consider  $p, q\in \mathbb D^4$ such that neither ${\bf p} $ nor ${\bf q}$ is the zero vector. Set $z= p_0 + {\bf i} |{\bf p}|$
 and $\zeta= q_0 + {\bf i} |{\bf q}|.$ Representation Formula gives 
\begin{align*}
 & \  2\|f (p)-f (q)\| \nonumber  \\
= &\|\left\{ (1-\frac{{\bf p}}{\|{\bf p}\|} {\bf i}) f(z) + (1+\frac{{\bf p}}{\|{\bf p}\|} {\bf i}) f(\bar z)
 - (1-\frac{{\bf q}}{\|{\bf q}\|} {\bf i}) f(\zeta) 
  - (1+\frac{{\bf q}}{\|{\bf q}\|} {\bf i}) f(\bar \zeta)    \right\} \| \nonumber \\
= &		 \|\left\{ (1-\frac{{\bf p}}{\|{\bf p}\|} {\bf i})( f(z)- f(\zeta) ) + (1+\frac{{\bf p}}{\|{\bf p}\|} {\bf i}) (f(\bar z)
   -   f(\bar \zeta)  ) \right. \nonumber  \\
& \left. + (\frac{{\bf p}}{\|{\bf p}\| } {\bf i} -\frac{{\bf q}}{\|{\bf q}\|}   {\bf i} ) ( f(\bar \zeta) -f(\zeta) ) 
       \right\} \|  \nonumber  \\
\leq & 	2 \|  f(z)- f(\zeta)  \|  + 2 \| f(\bar z)
   -   f(\bar \zeta)  \|   + \|\frac{{\bf p}}{\|{\bf p}\| }  -\frac{{\bf q}}{\|{\bf q}\|}  \| \| f(\bar \zeta) -f(\zeta)\|  
        \nonumber   \\
\leq & 	4  C_3
\sum_{k=1}^2  \omega_k (\| z  - \zeta  \|)   + \|\frac{{\bf p}}{\|{\bf p}\| }  -\frac{{\bf q}}{\|{\bf q}\|}  \|  C_3
\sum_{k=1}^2  \omega_k (\| \bar \zeta  - \zeta  \|).
	\end{align*}
 Note that  
$\|z-\zeta\| = \sqrt{ (p_0-q_0)^2 + (\|{\bf p}\|- \|{\bf q} \|)^2  }\leq \|p-q\|$ and as $\omega_1$ and $\omega_2$ 
are increasing functions then   
\begin{align}\label{equa12345}
2\|f (p)-f (q)\|\leq  & 	4 C_3 \left\{  
\sum_{k=1}^2   \left( \omega_k (\|p-q\|)      + \frac{1}{4} \|\frac{{\bf p}}{\|{\bf p}\| }  -\frac{{\bf q}}{\|{\bf q}\|}  \|   
  \omega_k (2  \|{\bf q} \| ) \right) \right\}.
	\end{align}
If $ 2\|{\bf q} \| \leq  \|p-q\| $ then $ \omega_k (2  \|{\bf q} \| ) \leq \omega_k (\|p-q\|)$, for $k=1,2$, and 
$$\displaystyle \|f (p)-f (q)\|\leq  	3 C_3 \sum_{k=1}^2     \omega_k (\|p-q\|).$$

On the other hand, if $\|p-q\| < 2\|{\bf q}\|$, from \eqref{equa12345},  we get 
	\begin{align*}
\frac{ \|f (p)-f (q)\|} {\displaystyle \sum_{k=1}^2    \omega_k (\|p-q\|)  } \leq  & 	2C_3 \left\{  
 1     + \frac{1}{4} \|\frac{{\bf p}}{\|{\bf p}\| }  -\frac{{\bf q}}{\|{\bf q}\|}  \|   
 \frac{\displaystyle   \sum_{k=1}^2 \omega_k (2  \|{\bf q} \| )}{ \displaystyle  \sum_{k=1}^2   \omega_k (\|p-q\|)  }   \right\}\\
  \\
	\leq & 	2C_3 \left\{  
 1     + \frac{1}{4} \|\frac{{\bf p}}{\|{\bf p}\| }  -\frac{{\bf q}}{\|{\bf q}\|}  \|   
 \sum_{k=1}^2 \frac{     \omega_k (2  \|{\bf q} \| )}{     \omega_k (\|p-q\|)  }   \right\}\\
  \\
	\leq & 	2C_3 \left\{  
 1     + \frac{1}{4} \|\dfrac{{\bf p}}{\|{\bf p}\| }  -\dfrac{{\bf q}}{\|{\bf q}\|}  \|  \dfrac{ 2\|{\bf q} \| }{   \|p-q\|} 
 \sum_{k=1}^2  \dfrac{    \dfrac{ \omega_k (2  \|{\bf q} \| )  }{  2  \|{\bf q} \| } }{   \dfrac{  \omega_k (\|p-q\|)} {  \|p-q\| }  }   \right\}\\
	\end{align*}
	As $\dfrac{\omega_k(t)}{t}$ is decreasing, for $k=1,2$, and   $\|p-q\| < 2\|{\bf q} \|$ then 
	\begin{align*}
	\sum_{k=1}^2  \dfrac{    \dfrac{ \omega_k (2  \|{\bf q} \| )  }{  2  \|{\bf q} \| } }{   \dfrac{  \omega_k (\|p-q\|)} {  \|p-q\| }  } \leq 2
\end{align*}
and 	
\begin{align*}
\|f (p)-f (q)\| \leq    & 	2C_3 \left\{  
 1     +    \|\dfrac{{\bf p}}{\|{\bf p}\| }  -\dfrac{{\bf q}}{\|{\bf q}\|}  \|  \dfrac{\|{\bf q} \| }{\|p-q\|}\right\}\displaystyle \sum_{k=1}^2    \omega_k (\|p-q\|).  
\end{align*}
It is easily seen that 
\begin{align*} 
    \|\dfrac{{\bf p}}{\|{\bf p}\| }  -\dfrac{{\bf q}}{\|{\bf q}\|}  \|  \dfrac{\|{\bf q} \| }{   \|p-q\|}   \leq  &    \|\dfrac{{\bf p}}{\|{\bf p}\| } -\dfrac{{\bf p}}{\|{\bf q}\| }  + \dfrac{{\bf p}}{\|{\bf q}\| } -\dfrac{{\bf q}}{\|{\bf q}\|}  \|  \dfrac{\|{\bf q} \| }{   \|  p -  q \|}	\\
				\leq  & \dfrac{|\|{\bf q} \| -\|{\bf p} \| |}{\|{\bf p}\|  \|{\bf q}\| }       \dfrac{\|{\bf p}{\bf q} \| }{   \|p-q\|}	+
		 \dfrac{\|{\bf p} - {\bf q} \| }{   \|p-q\|} 
\leq  2,
\end{align*}
and 
\begin{align*}
  \|f (p)-f (q)\|  \leq    & 	6C_3  \sum_{k=1}^2    \omega_k (\|p-q\|), 
\end{align*}
\end{enumerate}
which completes the proof by choosing $L=6C_3$.
 \end{proof}
\begin{rem} We have proved more, namely that for $\omega_1=\omega_2= \omega$ we are lead to 
$${}_{\bf i} \Lambda_{\omega  }(\mathbb D^4)\cap \mathcal {SR}(\mathbb D^4) =  G \Lambda_{\omega }(\mathbb D^4)\cap \mathcal {SR}(\mathbb D^4).$$  
Therefore, every slice regular function space associated to a majorant on a fix slice, it is also  associated to the same majorant on the four-dimensional unit ball and reciprocally.
\end{rem}
We proceed to describe some algebraic properties of the previously introduced functions sets.
\begin{prop}\label{prop0} 
Set $\bf{i} \in {\mathbb S}^2$. 
\begin{enumerate}
 \item  Given a  regular  majorant $\omega$, the sets ${}_{\bf i} \Lambda_{\omega} (\mathbb D^4) \cap \mathcal {SR}(\mathbb D^4)$ and $G \Lambda_{\omega} ({{\mathbb D^4}})\cap \mathcal {SR} (\mathbb D^4) $ are  quaternionic right linear spaces.
\item Let $\omega_1,\omega_2$ be  two regular majorants and let $f,g\in {}_{\bf i} \Lambda_{\omega_1, \omega_2} (\mathbb D^4)$. For every $a\in \mathbb H$ we have $f+g \in {}_{\bf i} \Lambda_{\omega_1, \omega_2} (\mathbb D^4)$ and 
$$fa \in {}_{\bf i} \Lambda_{\|a_1\|\omega_1 + \|a_2\|\omega_2, \ \|a_2\|\omega_1 + \|a_1\|\omega_2} (\mathbb D^4),$$ 
where $a=a_1+ a_2 {\bf j}$ with $a_1,a_2\in \mathbb C({\bf i})$ and ${\bf j} $ is orthogonal to ${\bf i}$.
\end{enumerate}
\end{prop}
\begin{proof}
\begin{enumerate}
\item Given $f,g\in {}_{\bf i} \Lambda_{\omega}(\mathbb D^4)\cap \mathcal {SR}(\mathbb D^4)$ and $a\in \mathbb H$ we see that    
\begin{align*}
 \|(fa+ g)(x)-(fa+ g)(y)\| \leq & \|a\| \|f(x)-f(y)\| + \|g(x)-g(y)\|\\
  \leq  & 
 C\omega (\| x - y \|) ,\  \   \forall x,y \in {\mathbb D}_{\bf i}.
\end{align*}
Similar inequalities are used to see that $G \Lambda_{\omega}(\mathbb D^4)\cap \mathcal {SR}(\mathbb D^4)$ is a quaternionic right linear space.
\item Given $f,g\in {}_{\bf i} \Lambda_{\omega_1, \omega_2} (\mathbb D^4)$ we have $f+g \in {}_{\bf i} \Lambda_{\omega_1, \omega_2} (\mathbb D^4)$ following a similar computation to the above. Denote $a=a_1+ a_2 {\bf j}$, $f\mid_{\mathbb D_{\bf i}}=f_1+ f_2 {\bf j}$,	$g\mid_{\mathbb D_{\bf i}}=g_1+ g_2 {\bf j}$ with ${\bf j}\in \mathbb S^2$ orthogonal to ${\bf i}$ and $a_1,a_2\in \mathbb C({\bf i})$ and $f_1,f_2,g_1,g_2 \in	Hol(\mathbb D_{\bf i})$.  
We obtain that  
$$fa \mid_{\mathbb D_{\bf i}} = (f_1  a_1 - f_2  \bar a_2 ) + (f_1  a_2 + f_2  \bar a_1){\bf j}$$ 
and 
	\begin{align*}
	  \| (f_1  a_1 - f_2  \bar a_2 )(x)- 
	(f_1  a_1 - f_2  \bar a_2 )(y) \| \leq & C_1 (\omega_1(\|x-y \|) \|a_1\| \\ 
	                                       &   + 
	\omega_2(\|x-y \|) \|a_2\| ),
\end{align*}
\begin{align*}
	   \| (f_1  a_2 + f_2  \bar a_1 )(x)- 
	(f_1  a_2 + f_2  \bar a_1 )(y) \| \leq &  C_2 (\omega_1(\|x-y \|) \|a_2\|  \\
	                                       &  + 
	\omega_2(\|x-y \|) \|a_1\| ).
\end{align*}
for all $x,y\in {\mathbb D}_{\bf i}$.

Note that picking out $\max\{\|a_1\|, \|a_2\|\}$ we can prove that   
$${}_{\bf i} \Lambda_{\|a_1\|\omega_1 + \|a_2\|\omega_2,
	\ \|a_2\|\omega_1 + \|a_1\|\omega_2} (\mathbb D^4)\subset 
	 {}_{\bf i} \Lambda_{ \omega_1 +\omega_2} (\mathbb D^4).$$ 
\end{enumerate}
\end{proof}
\begin{cor} Let $\omega_1$, $\omega_2$ be two regular majorants and $f\in {}_{\bf i} \Lambda_{\omega} (\mathbb D^4) \cap \mathcal N(\mathbb D^4)$. Then 
$$\|f\|_{ {}_{\bf i}\Lambda_{\omega_1}(\mathbb D^4)} = \|f\|_{ {}_{\bf k}\Lambda_{\omega_1}(\mathbb D^4)} = \|f\|_{ {}_{\bf i}\Lambda_{\omega_1, \omega_2}(\mathbb D^4)},$$ 
for all ${\bf k}\in \mathbb S^2$.
\end{cor}
\begin{proof}
Note that given $f \in \mathcal N(\Omega)$ there exists a sequence of real numbers $(a_{n})_{n=0}^{\infty}$, see \cite{CGS3,GS}, such that  
$f(q)= \sum_{n=0}^{\infty} q^n a_n$   for all $q\in \mathbb D^4$. Therefore for all $u\in \mathbb S^3$ one has that 
\begin{align*}
\|f(x) - f(y)\| = & \|u (   \sum_{n=0}^{\infty} x^n a_n -   \sum_{n=0}^{\infty} y^n a_n ) \bar u  \|  
              = \|   \sum_{n=0}^{\infty} (ux\bar u)^n a_n -   \sum_{n=0}^{\infty} (uy\bar u)^n a_n   \| \\
							=  &  \|f(u x\bar u) - f(uy\bar u)\| .\end{align*}
Choosing $u\in\mathbb S^3$ such that  $u{\bf i}\bar u = {\bf k}$ one obtains the first equality and for the second one we see that  
$f\mid_{\mathbb D_{\bf i}} =   f\mid_{\mathbb D_{\bf i}} + 0 {\bf j} $, i.e., $f_1=f\mid_{\mathbb D_{\bf i}}$ and $f_2=0$ in Definition \ref{def2}.
\end{proof}

\begin{rem} Note that if $f \in {}_{\bf i}\Lambda_{\omega}(\mathbb D^4)\cap C({\mathbb D}_{\bf i}, \mathbb H)$, then $f\mid_{{\mathbb D}_{\bf i}}$ can  be extended to a continuous function on $\overline{{\mathbb D}_{\bf i}}$. Similarly, if $f \in G\Lambda_{\omega}(\mathbb D^4)\cap C(\mathbb D^4, \mathbb H)$, then $f$ can be extended to a continuous function on $\overline{\mathbb D^4}$. 
\end{rem}
Now, we shall extend \cite[Theorems 1 and 2]{D} to slice regular function theory. 

\begin{prop} \label{prop01}  
\begin{enumerate}
\item Set $\bf{i}  \in \mathbb S^2$ and $f \in \mathcal{SR}(\mathbb D^4) \cap C(\overline{\mathbb D^4}, \mathbb H)$. Let $\omega$ and $\omega^2$ regular majorants. Then  
\begin{align*}  
\|f\|_{{}_{\bf i}\Lambda_{\omega}(\mathbb D^4)}^2  \asymp  &
\  \sup\{   \frac{    P[\|f_1\|^2] (x) - \|f_1(x)\|^2    }{ \omega( 1- \|x\|)^2 } \  \mid \ x\in \mathbb D_{\bf i}\}  \\ 
\\ 
& \  + \sup\{   \frac{  P[\|f_2\|^2] (x) - \|f_2(x)\|^2   }{ \omega( 1- \|x\|)^2 } \  \mid \ x\in \mathbb D_{\bf i}\}   \\
,
\end{align*}
where $f\mid_{{\mathbb D}_{\bf i}}=f_1+f_2 {\bf j}$ with ${\bf j}\in \mathbb S^2$ orthogonal to ${\bf i}$ and  $f_1,f_2 \in Hol({\mathbb D}_{\bf i})$.  
\item Set $\bf{i}  \in \mathbb S^2$ and $f \in \mathcal{SR}(\mathbb D^4) \cap C(\overline{\mathbb D^4}, \mathbb H)$. Let $\omega$ be a regular majorant. Then  
\begin{align*} 
 \|f\|_{{}_{\bf i}\Lambda_{\omega}(\mathbb D^4)}^2 \asymp N_1(f_1)^2 + N_1(f_2)^2   \asymp N_2(f_1)^2 + N_2(f_2)^2  \asymp N_3(f_1)^2 + N_3(f_2)^2   ,
\end{align*} 
where $f\mid_{{\mathbb D}_{\bf i}}=f_1+ f_2 {\bf j}$ with ${\bf j}\in \mathbb S^2$ orthogonal to ${\bf i}$ and $f_1,f_2 \in Hol({\mathbb D}_{\bf i})$. 	
\item Set $\bf{i}, {\bf k} \in \mathbb S^2$ and consider the regular majorants $\omega, \omega_1,\omega_2$. 
\begin{enumerate}
 \item  If $f\in {}_{\bf i} \Lambda_{\omega} (\mathbb D^4) \cap \mathcal{SR}(\mathbb D^4) \cap C(\overline{\mathbb D^4}, \mathbb H)$, then 
$$\|f\|_{ {}_{\bf i}\Lambda_\omega(\mathbb D^4)}\leq 2 \|f\|_{ {}_{\bf k}\Lambda_\omega(\mathbb D^4)}\leq 4 \|f\|_{ {}_{\bf i}\Lambda_\omega(\mathbb D^4)}.$$  
\item  If $f \in {}_{\bf i} \Lambda_{\omega , \omega }(\mathbb D^4)$, then 
$$
\|f\|_{ {}_{\bf i}\Lambda_{\omega, \omega}(\mathbb D^4)} =   \|f\|_{ {}_{\bf i}\Lambda_{\omega }(\mathbb D^4)} .$$
If $f \in {}_{\bf i} \Lambda_{\omega_1, \omega_2}(\mathbb D^4)$ then 
$$\|f\|_{ {}_{\bf i}\Lambda_{\omega_1 + \omega_2}(\mathbb D^4)}  \leq    \|f\|_{ {}_{\bf i}\Lambda_{\omega_1, \omega_2}(\mathbb D^4)}.$$ 
\end{enumerate}
\end{enumerate}
\end{prop}
\begin{proof}
\begin{enumerate}
\item By \eqref{Result1OfD} and the fact that $\alpha \asymp \beta$ and $\delta \asymp \gamma$ imply $\alpha^2  +  \delta^2  \asymp \beta^2 + \gamma^2$. Also, the application of inequalities    
$$\|f\|_{{}_{\bf i}\Lambda_{\omega}(\mathbb D^4)}^2  \leq   \|f_1 \|_{{}_{\bf i}\Lambda_{\omega}(\mathbb D_{\bf i})}^2 + \|f_2\|_{{}_{\bf i}\Lambda_{\omega}(\mathbb D_{\bf i})}^2 \leq 2 \|f\|_{{}_{\bf i}\Lambda_{\omega}(\mathbb D^4)}^2.$$ 
\item By \eqref{Result2OfD} and the properties stated above.  
\item Fact (a) follows from direct computations and the idea of the ensuing Fact (b) is the following:
\begin{align*}
\|f\|_{ {}_{\bf i}\Lambda_{\omega , \omega }(\mathbb D^4)}^2 =  & \sup\{    \frac{  \|f_1(x)- f_1(y) \|^2  }{ \omega (\|x-y\|)^2} +  \frac{  \|f_2(x)- f_2(y) \|^2  }{ \omega (\|x-y\|)^2}  \      \mid  \   x,y \in \mathbb D_{\bf i}, \  x\neq y\} \\
 =  &\sup\{    \frac{  \|f (x)- f (y) \|^2  }{ \omega (\|x-y\|)^2}  \      \mid  \   x,y \in \mathbb D_{\bf i}, \  x\neq y\} = \|f\|_{ {}_{\bf i}\Lambda_{\omega  }(\mathbb D^4)}^2.
\end{align*}
\end{enumerate}
\end{proof}

\begin{cor} 
\begin{enumerate}
\item Set $\bf{i}  \in \mathbb S^2$ and $f \in \mathcal{N}(\mathbb D^4) \cap C(\overline{\mathbb D^4}, \mathbb H)$.  If $\omega$ and $\omega^2$ are  regular majorant, then  
\begin{align*}  \sup\{   \frac{  P[\|f\mid_{\mathbb D_{\bf i}} \|^2] (x) - \|f\mid_{\mathbb D_{\bf i}}(x)\| ^2 }{ \omega( 1- \|x\|)  } \  \mid \ x\in \mathbb D_{\bf i}\}  \asymp 
\|f\|_{{}_{\bf i}\Lambda_{\omega}(\mathbb D^4)} .
\end{align*}
\item Suppose  $\bf{i}  \in \mathbb S^2$, $f \in \mathcal{N}(\mathbb D^4) \cap C(\overline{\mathbb D^4}, \mathbb H)$ and $\omega$ a regular majorant. Then  
\begin{align*} 
 \|f\|_{{}_{\bf i}\Lambda_{\omega}(\mathbb D^4)}^2 \asymp N_1(f\mid_{\mathbb D_{\bf i}})   \asymp N_2(f\mid_{\mathbb D_{\bf i}} )    \asymp N_3(f\mid_{\mathbb D_{\bf i}})   ,
\end{align*}
\end{enumerate}
\end{cor}
\begin{proof}
Both facts follow from $f_1=f_1\mid_{\mathbb D_{\bf i}}$ and $f_2=0$ in Definition \ref{def2} since  $f \in \mathcal{N}(\mathbb D^4)$ and $\displaystyle f(q)= \sum_{n=0}^{\infty} q^n a_n$ for all $q\in \mathbb D^4$ iff $a_n\in \mathbb R$ for all $n$, see \cite{CGS3,GS}.
\end{proof}
As the function sets given in Definitions \ref{def1}, \ref{def2} and \ref{def3} depend of unit vectors the following proposition shows some  relationships between them.  
\begin{prop} \label{prop1}  Set $\bf{i} \in \mathbb S^2$ and consider the regular majorants $\omega, \omega_1,\omega_2$. 
\begin{enumerate}
\item If $f\in {}_{\bf i} \Lambda_{\omega} (\mathbb D^4)$ then  $\|f\|,  \| f\pm {\bf i}f{\bf i} \|\in  {}_{\bf i} \Lambda_{\omega} (\mathbb D^4)$. 
\item  ${}_{\bf i} \Lambda_{\omega}(\mathbb D^4)={}_{\bf i}\Lambda_{\omega, \omega}(\mathbb D^4)$.
\item  ${}_{\bf i}\Lambda_{\omega_1, \omega_2}(\mathbb D^4) \subset {}_{\bf i} \Lambda_{\omega_1+\omega_2}(\mathbb D^4)$.
\item ${}_{\bf i} \Lambda_{\omega}(\mathbb D^4) \cap \mathcal {SR}(\mathbb D^4) = {}_{\bf j} \Lambda_{\omega}(\mathbb D^4) \cap \mathcal {SR}(\mathbb D^4)$.	
\end{enumerate}
\end{prop}
\begin{proof}
\begin{enumerate}
\item  Given $f\in {}_{\bf i} \Lambda_{\omega} (\mathbb D^4)$ set $f=f_1+f_2{\bf j}$ on ${\mathbb D}_{\bf i}$, where ${\bf j}\in \mathbb S^2$ is orthogonal to ${\bf i}$  and $f_1,f_2 : {\mathbb D}_{\bf i} \to {\mathbb D}_{\bf i} $. Then we see that  
\begin{align}\label{components}
2f_1 = f-{\bf i}f {\bf i }, \ \  2f_2 {\bf j} = f+{\bf i}f {\bf i }, \  \ \textrm{on } \ {\mathbb D}_{\bf i}.
\end{align}
From inequalities 
$$|\ \|f(x)\|-  \|f(y)\| \  | \leq  \|f(x)- f(y)\|,$$
$$\max\{ |\ \|f_1(x)\|-  \|f_1(y)\| \  |, \  |\ \|f_2(x)\|-\|f_2(y)\| \  | \}\leq \|f(x)- f(y)\|,$$
for all $x,y \in {\mathbb D}_{\bf i}$, it follows that $\|f\|, \|f\pm {\bf i}f{\bf i}\|\in {}_{\bf i} \Lambda_{\omega} (\mathbb D^4)$. 
\item[2. and 3.] With the previous notation the identity 
\begin{align}\label{norm_components}   
\|f(x)- f(y)\|^2  =  \|f_1(x)-f_1(y)\|^2  +  \|f_2(x)-f_2(y)\|^2 ,\quad \forall x,y\in {\mathbb D}_{\bf i},
\end{align}  
holds, allowing us to see that 
$${}_{\bf i} \Lambda_{\omega}(\mathbb D^4)={}_{\bf i}\Lambda_{\omega, \omega}(\mathbb D^4)$$ 
and 
$${}_{\bf i}\Lambda_{\omega_1, \omega_2}(\mathbb D^4) \subset {}_{\bf i} \Lambda_{\omega_1+\omega_2}(\mathbb D^4).$$
\item[4.]  Given $f \in {}_{\bf i} \Lambda_{\omega}(\mathbb D^4) \cap \mathcal {SR}(\mathbb D^4)$ and ${\bf j}\in \mathbb S^2$, according to the Representation Formula, we have for every  $ x_0+ x_1{\bf j}, y_0+ y_1{\bf j} \in  {\mathbb D}_{{\bf j}}$ with $x_0,x_1, y_0,y_1 \in\mathbb R$ 
\begin{align*}
\|f(x_0+ x_1{\bf j}) - f(y_0+ y_1{\bf j})\| = \\ \frac {1}{2}\|( 1- {\bf j} {\bf i})(f(x) - f(y)) 
+   ( 1 + {\bf j}  {\bf i})  (f(\bar x) - f(\bar y ) ) \|  \\
\leq  \|  f( x ) - f(y )\| +   \| f(\bar x) - f(\bar y) \| \leq 2 C \omega (\|x-y\|) \\
\leq   2 C \omega (\| (x_0+ x_1{\bf j}) - (y_0+ y_1{\bf j})\|),  
\end{align*}
where $x = x_0+ x_1{\bf i}$ and $y=y_0+ y_1{\bf i}$. 
\end{enumerate}
\end{proof}
\begin{rem}
Repeating the computations presented in Propositions \ref{prop0} and \ref{prop1} enables us to see that ${}_{\bf i} \Lambda_{\omega}(\mathbb S^3)$, ${}_{\bf i}\Lambda_{\omega_1, \omega_2}(\mathbb S^3) $ and $G \Lambda_{\omega}(\mathbb S^3)$ have similar properties of ${}_{\bf i} \Lambda_{\omega}(\mathbb D^4)$, ${}_{\bf i}\Lambda_{\omega_1, \omega_2}(\mathbb D^4) $ and $G \Lambda_{\omega}(\mathbb D^4)$, respectively, and that is why they are omitted. 
\end{rem}
Let ${\omega}$, $\omega_1$ and $\omega_2$ be regular majorant. The next propositions characterize the elements of $\mathcal {SR}(\mathbb D^4)\cap {}_{\bf i} \Lambda_{\omega}(\mathbb D^4)$ and of $\mathcal {SR}(\mathbb D^4)\cap {}_{\bf i} \Lambda_{\omega_1,\omega_2} (\mathbb D^4)$, which extend results contained in \cite{D,P}.   
\begin{prop}     
\begin{enumerate}
\item  Set $f\in \mathcal {SR}(\mathbb D^4)$. Then $f\in {}_{\bf i} \Lambda_{\omega}(\mathbb D^4)$ if and only if 
$$\|f'(x) \pm {\bf i} f'(x) {\bf i}\|\leq C\frac{\omega (1-\|x\|)}{1-\|x\|},$$
for $C$ independent of $x\in {\mathbb D}_{\bf i}$.
\item  Set $f\in \mathcal {SR}(\mathbb D^4)\cap C(\overline{\mathbb D^4}, \mathbb H)$. Then    
$$\frac{1}{2}(1-\|x\|)\|f'(x) \pm {\bf i} f'(x) {\bf i}\|   +   \|f(x) \pm {\bf i} f(x) {\bf i}\| \leq 2 M_x,$$
for $C$ independent of $x\in {\mathbb D}_{\bf i}$, where 
$$M_x = \sup \{\|f(y)\| \ \mid \  	\|y -  x\|\leq 1-\|x\|, \ y\in {\mathbb D}_{\bf i}\}.$$
\item  Set $f\in \mathcal {SR}(\mathbb D^4)$. Then $f\in {}_{\bf i} \Lambda_{\omega} (\mathbb D^4)$ if and only if 
$$\|f'(x)\|\leq C\frac{\omega (1-\|x\|)}{1-\|x\|},$$
for $C$ independent of $x\in {\mathbb D}_{\bf i}$.
\item  Set $ f\in \mathcal {SR}(\mathbb D^4)\cap {}_{\bf i} \Lambda_{\omega_1, \omega_2} (\mathbb D^4)$. Then  
$$\|f '(x) \|\leq C \left( \frac{ \sqrt{ \omega_1 (1-\|x\|)^2+\omega_2 (1-\|x\|)^2 } } 
{(1-\|x\|)}\right),$$
for $C$ independent of $x\in {\mathbb D}_{\bf i}$.
\item Set $f\in \mathcal{SR}(\mathbb D^4)\cap C(\overline{\mathbb D^4},\mathbb H)$. Then   
\begin{align*}
\frac{1}{4}(1-\|x\|)^2\|f'(x)\|^2 +   \|f(x)\|^2     \leq  &  (\|x\|-1) (\ \|f_1'(x)\|\|f_1(x)\|  + 
\|f_2'(x)\|\|f_2(x)\|  \ )  \\ 
 & + M_{1,x}^2+ M_{2,x}^2,\end{align*} 
for $ x\in {\mathbb D}_{\bf i}$, where $f\mid_{{\mathbb D}_{\bf i}}= f_1+ f_2{\bf j}$ with $f_1, f_2\in Hol({\mathbb D}_{\bf i}) \cap C(\overline{\mathbb D}_{\bf i}, \mathbb C(\bf i))$ and 
$$M_{k,x} = \sup\{\|f_k(y)\|  \ \mid  \   \|y-x\| \leq 1-\|x\|, \ y\in {\mathbb D}_{\bf i}\},$$
for $k=1,2$.
\item Consider $f\in \mathcal{SR}(\mathbb D^4)\cap C(\overline{\mathbb D^4}, \mathbb H)$. Let us introduce one more piece of notation:
$$M= \sup\{ \|f(w) \|  \ \mid \ w\in {\mathbb D}_{\bf i}\}.$$
If there exists ${\bf i}\in \mathbb S^2$ and a regular majorant $\omega$ such that 
$$\|M^2- \overline{f(x)} f(\tilde{x})\| \leq C(1+\|x\|)\omega(1-\|x\|),$$
for $C$ independent of $x\in {\mathbb D}_{\bf i}\setminus Z_{f'}$, where 
\begin{align*} 
  \tilde{x} = & {(\overline{f (x)})}^{-1}T_g (f' (x)^{-1}x f' (x) )
\overline{ f (x)},    \\
 g(x)=  &  1-\overline{f(x)} * f(x),
\end{align*} 
where  $T_g (q) = (g^c(q))^{-1}q g^c(q)$  for all $q \in \mathbb D^4$  such that $g^ s(q)\neq 0 $,
then 
$$\|f'(x)\| \leq \frac{C}{M} \frac{\omega(1-\|x\|)}{ 1-\|x\|};$$
that is, $f\in {}_{\bf i}\Lambda_{\omega}(\mathbb D^4)$, see Fact 3. of the present proposition.
\end{enumerate}
\end{prop}
\begin{proof}
\begin{enumerate}
\item With the Splitting  Property in mind, consider $f_1,f_2\in Hol({\mathbb D}_{\bf i})$ such that $f\mid_{{\mathbb D}_{\bf i}} = f_1+f_2 {\bf j}$ where ${\bf j}\in \mathbb S^2$ is orthogonal to ${\bf i}$. From Fact 2. of Proposition \ref{prop1} one has that $f\in {}_{\bf i} \Lambda_{\omega} 
(\mathbb D^4)$ if and only if  $f_1,f_2\in \Lambda_{\omega, \omega} 
({\mathbb D}_{\bf i})$, that is, 
$$\|f '(x) \pm {\bf i} f '(x) {\bf i}\|\leq C\frac{\omega (1-\|x\|)}{1-\|x\|},$$
for $C$ independent of $x\in {\mathbb D}_{\bf i}$, where we use \cite[Lemma 1]{P}. The result is obtained from \eqref{components} applied to $f'$. 
\item Splitting  Property implies that 
$f\mid_{{\mathbb D}_{\bf i}} = f_1+f_2 {\bf j}$ where  $f_1,f_2\in Hol({\mathbb D}_{\bf i})$ and ${\bf j}\in \mathbb S^2$ is orthogonal to ${\bf i}$. From Fact 2. of Proposition \ref{prop1} one has that $f\in {}_{\bf i} \Lambda_{\omega}(\mathbb D^4)$ if and only if  $f_1,f_2\in \Lambda_{\omega, \omega}({\mathbb D}_{\bf i})$. Applying \cite[Lemma 2]{P} to the complex components of $f\mid_{{\mathbb D}_{\bf i}}$ and using \eqref{components} in $f$ and $f'$ to finish the proof.
\item It follows from Fact 1. commbining with the following consequence of the parallelogram identity:
$$4\|f'(x)\|^2 = \|f'(x) + {\bf i} f'(x) {\bf i}\|^2+ \|f'(x) - {\bf i} f'(x) {\bf i}\|^2,$$
for all $x\in {\mathbb D}_{\bf i}$.
\item  Kipping in mind the Splitting Property, set $f_1,f_2\in Hol({\mathbb D}_{\bf i})$ such that $f\mid_{{\mathbb D}_{\bf i}} = f_1+f_2 {\bf j}$ where ${\bf j}\in \mathbb S^2$ is orthogonal to ${\bf i}$. From Fact 2. of Proposition \ref{prop1} one has that $f\in {}_{\bf i} \Lambda_{\omega}(\mathbb D^4)$ if and only if  $f_1,f_2\in \Lambda_{\omega, \omega}({\mathbb D}_{\bf i})$, i.e., 
$$\|f_k'(x) \|\leq C \frac{\omega_k (1-\|x\|)}{1-\|x\|},$$
for $C$ independent of $x\in {\mathbb D}_{\bf i}$ and for $k=1,2$, see \cite[Lemma 1]{P}. Applying \eqref{norm_components} to $f'$ yields  
$$\|f'(x)\|^2\leq C^2 \left( \frac{\omega_1 (1-\|x\|)^2+\omega_2 (1-\|x\|)^2}{(1-\|x\|)^2}\right).$$
\item  Given $f\mid_{{\mathbb D}_{\bf i}}= f_1+ f_2{\bf j}$ with $f_1, f_2\in Hol({\mathbb D}_{\bf i}) \cap C(\bar {\mathbb D}_{\bf i}, 
\mathbb C(\bf i))$ from \cite[Lemma 2]{P} we see that    
$$\frac{1}{2} (1- \|x\|) \|f_k'(x)\|+\|f_k(x)\|\leq M_{k,x},  $$ for $  x\in {\mathbb D}_{\bf i}$,
where   $M_{k,x} = \sup\{ \|f_k (y)\|  \ \mid  \   \|y-x\| \leq 1-\|x\|, \ y\in {\mathbb D}_{\bf i}\}$ for $k=1,2$. 

Therefore 
$$\frac{1}{4} (1- \|x\|)^2\|f_k'(x)\|^2+ (1- \|x\|)\|f_k'(x)\| \|f_k(x)\|+ \|f_k(x)\|^2\leq M_{k,x}^2,$$
for $k=1,2$. Adding terms in the previous inequalities and using \eqref{norm_components} applied to $f$ and $f'$, the main result follows.
\item Let $f\in \mathcal{SR}(\mathbb D^4)\cap C(\overline{\mathbb D^4}, \mathbb H)$. With the notation $F(x)=\dfrac{f(x)}{M}$ for all $x\in \mathbb D^4$ rewrite the hypothesis as follows: 
$$\|1- \overline{F(x)} F(\tilde{x})\| \leq \frac{C}{M^2}(1+\|x\|)\omega(1-\|x\|),$$
for $C$ independent of $x$, or equivalently 
$$ \frac{\|1- \overline{F(x)} F(\tilde{x})\|}{1- \|x\|^2} \leq \frac{C   }{M^2}\  \frac{\omega(1-\|x\|) }{  1- \|x\|}.$$
Equation (3.10) of \cite[Theorem 3.7 (Schwarz-Pick lemma)]{BS} shows that if $f:\mathbb D^4 \to \mathbb D^4$ is a slice regular function
and $q_0 \in \mathbb D^4$ implies that
$$\|\partial_c f \ast  (1 - f (q_0) \ast f (q))^{ -\ast} \|_{\mid_{q_0}} \leq \frac{ 1}{1 - |q_0|^2} .$$
This finally yields  
$$\|F'(x)\| \leq \frac{C}{M^2} \frac{\omega(1-\|x\|)}{ 1-\|x\|},$$ 
that is 
$$\|f'(x)\| \leq \frac{C}{M} \frac{\omega(1-\|x\|)}{ 1-\|x\|}$$
and Fact 6. is proved
\end{enumerate}
\end{proof}

\begin{cor}     
\begin{enumerate}
\item Set $f\in \mathcal {SR}(\mathbb D^4)$. Then $f\in {}_{\bf i} \Lambda_{\omega}(\mathbb D^4)$ if and only if 
$$\|f '(q)  \|\leq C\frac{\omega (1-\|q\|)}{1-\|q\|},$$
for some $C$ independent of the choice of $q\in {\mathbb D}^4$.
\item   Set  $ f\in \mathcal {SR}(\mathbb D^4)\cap C(\overline{\mathbb D^4}, \mathbb H)$ then 
\begin{align*} 
 & \ \ \frac{1}{2}(1-\|q\|)\|f'(q)\|   +   \|f(q)  \|  \\
	\leq  &   \ \    \sup \{\|f(y)\| \ \mid \  	\sqrt{ (\Re y - q_0)^2 +(\Im y - |{\bf q}| )^2} \leq 1-\|q\|, \ y \in {\mathbb D}_{\bf i}\}  \\
	       &  +   \sup \{\|f(y)\| \ \mid \  	\sqrt{ (\Re y - q_0)^2 +(\Im y  + |{\bf q}| )^2} \leq 1-\|q\|, \ y \in {\mathbb D}_{\bf i}\}.  
				\end{align*}
\item  If $ f\in \mathcal {SR}(\mathbb D^4)\cap {}_{\bf i} \Lambda_{\omega_1,
\omega_2} (\mathbb D^4)$ then  
$$\|f '(q) \|\leq C \left(\frac{\omega_1 (1-\|q\|)+\omega_2 (1-\|q\|) } 
{(1-\|q\|)}\right),$$
for $C$ independent of the choice of $q\in {\mathbb D}^4$.
\end{enumerate}
\end{cor}
\begin{proof}
\begin{enumerate}
\item The sufficient condition is a direct consequence of the previous proposition and 
$$\|f'(x) \pm {\bf i} f'(x) {\bf i}\|\leq  2 \|f'(x) \|, \quad \forall x\in \mathbb D_{\bf i}.$$  
On the other hand, let $f\in {}_{\bf i} \Lambda_{\omega}(\mathbb D^4)\cap \mathcal {SR}(\mathbb D^4)$ and $q\in \mathbb D^4$ such that ${\bf q}$ is not the vector zero. Applying Representation Formula we deduce that  
\begin{align*}
2\|f '(q)\| \leq & 2\|f'(x)\| + 2\|f '(\bar x)\| \leq  \|f'(x) + {\bf i} f'(x) {\bf i}\|  +  \|f'(x) - {\bf i} f'(x) {\bf i}\|+ \\  
				              &   \|f'(\bar x) + {\bf i} f'(\bar x) {\bf i}\| + \|f'(\bar x) - {\bf i} f'(\bar  x) {\bf i}\|								
						\leq \\ & 4 C\frac{\omega (1-\|x\|)}{1-\|x\|} = 4 C\frac{\omega (1-\|q\|)}{1-\|q\|},
\end{align*}
where $x= q_0 + {\bf i}|{\bf q}|\in {\mathbb D}_{\bf i}.$
\item For fixed $q\in\mathbb D^4$ such that ${\bf q}$ is not the zero vector, let $x=x_0 + {\bf i}|{\bf q}| \in \mathbb D_{\bf i}$. By the Representation Formula we obtain  
\begin{align*} 
& \ \frac{1}{2}(1-\|q\|)\|f'(q)\|   +   \|f(q)  \| \\
\leq & ( \  \frac{1}{2}(1-\|x\|) \|f'(x)\| + \|f(x)\| \ )  + (\ \frac{1}{2}(1-\|\bar x\|) \|f'(\bar x)\| + \|f(\bar x)\| \ )    \\
\leq & ( \  \frac{1}{4}(1-\|x\|)\|f'(x) + {\bf i} f'(x) {\bf i}\| + \frac{1}{2} \|f(x) + {\bf i} f(x) {\bf i}\| \|  \  )   \\
     &  + ( \  \frac{1}{4}(1-\|  x\|) \|f'(x) - {\bf i} f'(x) {\bf i}\|  + \frac{1}{2} \|f(x) - {\bf i} f(x) {\bf i}\|  \ )    \\
 & + ( \  \frac{1}{4}(1-\|\bar x\|) \|f'(\bar x) + {\bf i} f'(\bar  x) {\bf i}\| +  \frac{1}{2} \|f (\bar  x) + {\bf i} f(\bar x) {\bf i}\| \| \  ) \\
 &  + ( \  \frac{1}{4}(1-\|  \bar  x\| ) \|f'(\bar  x) - {\bf i} f'(\bar  x) {\bf i}\|  + \frac{1}{2} \|f(\bar  x) - {\bf i} f(\bar  x) {\bf i}\|  \ )   \\
 \leq  & \    M_x +  M_{\bar  x},
\end{align*}
where 
$$M_x =\{\|f(y)\| \ \mid \  	\sqrt{ (\Re y - q_0)^2 +(\Im y - |{\bf q}| )^2} \leq 1-\|q\|, \ y \in {\mathbb D}_{\bf i}\}.$$
$$M_{\bar x} =\{\|f(y)\| \ \mid \  	\sqrt{ (\Re y - q_0)^2 +(\Im y + |{\bf q}| )^2} \leq 1-\|q\|, \ y \in {\mathbb D}_{\bf i}\}.$$
\item It follows in a similar way like Fact 1.
\end{enumerate}
\end{proof}
Here are some properties of $P_{{\bf i}}$ and its dependence on ${\bf i}$.
\begin{prop} Let $\bf{i} \in \mathbb S^2$. The following items hold
\begin{enumerate}
\item  Given $r\in \mathbb S^3$ write $T_r({q}):=rq\bar r $ for all $q\in \mathbb D^4$.
Then 
$$P_{T_r({\bf i})} [u] (q) = P_{\bf i} [u\circ T_{r}] (T_{r}^{-1}(q)),$$ 
for all $u \in C^1 ( \mathbb S_{T_r({\bf i})}, \mathbb R)$.     
\item Given $f \in  C(\overline{\mathbb D^4}, \mathbb H)$ we have on ${\mathbb D}_{\bf i}$ that 
$$\displaystyle P_{\bf i }[\|f\pm {\bf i} f{\bf i}\| ] \leq  2 P_{\bf i }[\|f\|] \leq P_{\bf i }[\|f- {\bf i} f{\bf i}\|] + P_{\bf i }[\|f+ {\bf i} f{\bf i}\|].$$
\item  Given $f\in  \mathcal {SR}({\mathbb D^4})\cap C(\overline{\mathbb D^4}, \mathbb H)$ and ${\bf j}\in \mathbb S^2$ it follows that 
$$\displaystyle\frac{1}{2\pi}\int_{0}^{2\pi} \|(x- e^{{\bf j} t})^{-*2}*  f(e^{{\bf j} t})  \| (1-\|x\|^2) dt  \leq  2 P_{\bf i} [\|f\|](x), $$ for all $x\in {\mathbb D}_{\bf i}$.
\end{enumerate}
\end{prop}
\begin{proof}
\begin{enumerate}
\item Since $e^{r{\bf i} \bar r t}= re^{ {\bf i}  t} \bar r $ we have  
\begin{align*}
P_{T_r({\bf i})} [u] (q) = \frac{1}{2\pi }\int_{0}^{2\pi} u(e^{r{\bf i} \bar r t} ) \frac{1-\|q\|^2}{ \|q- e^{r{\bf i} \bar r t} \|^2 } dt ,
\end{align*}
and it may be concluded that 
\begin{align*}
P_{T_r({\bf i})} [u] (q) = \frac{1}{2\pi }\int_{0}^{2\pi} u\circ T_r (e^{{\bf i}  t} ) 
\frac{1-\|\bar rq  r \|^2}{ \|\bar r q r- e^{ {\bf i} t} \|^2}dt.
\end{align*}
\item It is due to  \eqref{components} and \eqref{norm_components}
\item Let $x\in {\mathbb D}_{\bf i}$ and ${\bf j}\in \mathbb S^2$. According to the Representation Formula and the established continuity we have 
$$(x- e^{{\bf j} t})^{-*2}) *  f(e^{{\bf j} t})     = \frac{1}{2} [(1+{\bf j } {\bf i})(x- e^{ - {\bf i} t})^{- 2} f(e^{{-\bf i } t})  +  
(1-{\bf j } {\bf i}) (x- e^{  {\bf i} t})^{- 2} f(e^{{\bf i } t})].$$
Therefore 
\begin{align*} 
\|(x- e^{{\bf j} t})^{-*2}) *  f(e^{{\bf j} t})  \|(1-\|x\|^2)   = & 
\|x- e^{ - {\bf i} t}\|^{- 2} \|f(e^{{-\bf i } t})\| (1-\|x\|^2)      \\
 & + \|x- e^{  {\bf i} t}\|^{- 2} \|f(e^{{\bf i } t}) \| (1-\|x\|^2)     
\end{align*}
and
 \begin{align*} &  \frac{1}{2\pi}\int_{0}^{2\pi} \|(x- e^{{\bf j} t})^{-*2} *  f(e^{{\bf j} t})  \| (1-\|x\|^2) dt \\
 \leq  &  \frac{1}{2\pi}\int_{0}^{2\pi}  \|x- e^{ - {\bf i} t}\|^{- 2} \|f(e^{{-\bf i } t})\| (1-\|x\|^2) dt\\
 &     +   \frac{1}{2\pi}\int_{0}^{2\pi}  \|x- e^{  {\bf i} t}\|^{- 2} \|f(e^{{\bf i } t}) \| (1-\|x\|^2)dt   \\
= & 2 P_{\bf i} [\|f\|](x).   
\end{align*}
\end{enumerate} 
\end{proof}

\begin{prop} Set $f\in \mathcal{SR}(\mathbb D^4)\cap C(\overline{\mathbb D^4})$.
\begin{enumerate} 
\item Let $\omega$ be a regular majorant  such that  $\|f\|\in  {}_{\bf i}\Lambda_{\omega}(\mathbb S^3)$. Then $f\in {}_{\bf i}\Lambda_\omega (\mathbb D^4)$ if and only if 
$$P_{\bf{i}}(\|f\pm {\bf i} f {\bf i} \|)(x)-\|f (x)\pm {\bf i} f(x) {\bf i} \| \leq C \omega (1-\|x\|), \quad \forall x\in {\mathbb D}_{
\bf i}.$$
\item Let $\omega_1$, $\omega_2$ be two regular majorant  such that $\|f\|\in  {}_{\bf i}\Lambda_{\omega_1,\omega_2}(\mathbb S^3)$. Then $f\in {}_{\bf i}\Lambda_{\omega_1,\omega_2}(\mathbb D^4)$ if and only if 
\begin{align*}
P_{\bf{i}}(\|f- {\bf i} f {\bf i} \|)(x)-\|f (x)- {\bf i} f(x) {\bf i} \| \leq & C  \omega_1 (1-\|x\|),\\
P_{\bf{i}}(\|f+ {\bf i} f {\bf i} \|)(x)-\|f (x)+ {\bf i} f(x) {\bf i} \| \leq & C  \omega_2 (1-\|x\|), \quad \forall x\in {\mathbb D}_{
\bf i}.
\end{align*}
\end{enumerate}
\end{prop}
\begin{proof}
\begin{enumerate}
\item  Combining \eqref{components} with Proposition \ref{prop1} and Theorem B of Section 2 completes the proof. 
\item  It follows from identities \eqref{components} and Theorem B of Section 2. 
\end{enumerate}
\end{proof}
 
\begin{cor}\label{cor56} Consider $f\in \mathcal{SR}(\mathbb D^4)\cap C(\overline{\mathbb D^4})$. 
\begin{enumerate}
\item  Let $\omega$ be a regular majorant  such that $\|f\|\in  {}_{\bf i}\Lambda_{\omega}(\mathbb S^3)$. If $f\in {}_{\bf i}\Lambda_{\omega} (\mathbb D^4)$ and 
$q\in \mathbb D^4$ satisfies  $$\langle q, e^{{\bf i} t} \rangle  \leq q_0 \cos t \pm |{\bf q}| \sin t,  \quad \forall t\in [0,2\pi],$$
then  
\begin{align*}
 P_{\bf{i}}(\|f  \|)(q)- 2\|f (q_0 \pm {\bf i}|{\bf q}|) \| \leq C \omega (1-\|q\|).
\end{align*}
\item Let $\omega_1$, $\omega_2$ be two regular majorant  such that $\|f\|\in  {}_{\bf i}\Lambda_{\omega_1,\omega_2}(\mathbb S^3)$. If $f\in {}_{\bf i}\Lambda_{\omega_1,\omega_2}(\mathbb D^4)$  and 
$q\in \mathbb D^4$ holds  $$\langle q, e^{{\bf i} t} \rangle  \leq q_0 \cos t \pm |{\bf q}| \sin t,  \quad \forall t\in [0,2\pi],$$
then  
 \begin{align*}
 P_{\bf{i}}(\|f  \|)(q) - 2\|f (q_0 \pm {\bf i}|{\bf q}|) \| \leq  C \left(\omega_1 (1-\|q\|) + \omega_2 (1-\|q\|) \right).
\end{align*}
\end{enumerate} 
\end{cor} 
 \begin{proof}
\begin{enumerate} 
\item 
Let $x= q_0 \pm {\bf i}|{\bf q}|$. It follows easily that  
$ |q-e^{{\bf i}}| \geq |x- e^{{\bf i} t}|$ for all $t\in [0,2\pi]$. 
Of course,  
\begin{align*}
 2P_{\bf{i}}(\|f  \|)(q) \leq  &    2 P_{\bf{i}}(\|f  \|)(x)   \leq   P_{\bf{i}}(\|f+  {\bf i} f {\bf i} \|)(x)  +  P_{\bf{i}}(\|f- {\bf i} f {\bf i} \|)(x)  \\
\leq  &   \|f (x) + {\bf i} f(x) {\bf i} \| + \|f (x)- {\bf i} f(x) {\bf i} \| + 2 C \omega (1-\|x\|).
\end{align*}
Therefore 
\begin{align*}
 P_{\bf{i}}(\|f  \|)(q)\leq 2\|f (q_0 \pm {\bf i}|{\bf q}|) \| + C \omega (1-\|q\|).
\end{align*}
\item Similar arguments to those above.
\end{enumerate}
\end{proof}

\section{Conclusions and future works} 
In summary, characterizations of the Lipschitz type spaces of slice regular functions in the unit ball of the skew-field of quaternions with prescribed modulus of continuity, despite the non-commutativity of quaternions, are presented. The main results go on to the global case. Importantly, the present findings suggest the possibility to extend the study to the theory of slice monogenic functions associated to Clifford algebras, as a good starting point for further research.

\subsection*{Acknowledgment}
The authors greatly appreciate the many helpful comments and suggestions made by the reviewers.

\subsubsection*{Funding} Instituto Polit\'ecnico Nacional (grant number SIP20211188, SIP20221274) and CONACYT.}

\subsection*{Compliance with ethical standards}

\subsubsection*{Conflict of interest} The authors declare that they have no conflict of interest regarding the publication of this paper.

  \end{document}